\documentclass[12pt,twoside,reqno]{amsart}

\usepackage[OT1]{fontenc}
\usepackage{type1cm}

\usepackage{comment}
\usepackage{enumerate}

\usepackage{amsthm}
\RequirePackage{amsmath,amsfonts,amssymb,amsthm}

\RequirePackage[dvips]{graphicx}

\usepackage{psfrag}
\usepackage[usenames]{color}
  \numberwithin{equation}{section}

\usepackage[active]{srcltx}  

\usepackage{booktabs}

  \newcommand{\N}{\mathbb{N}}         
  \newcommand{\R}{\mathbb{R}}         
  \newcommand{\Z}{\mathbb{Z}}

  \newcommand{\EE}{\mathbb{E}}        
  \newcommand{\PP}{\mathbb{P}}        
  \newcommand{\leb}{\mathcal{L}}      
  \newcommand{\supp}{\text{supp}}        
  
  \newcommand{\dimloc}{\text{dim}}
  \newcommand{\ldimloc}{\underline{\dim}}
  \newcommand{\udimloc}{\overline{\dim}}

  \newcommand{\e}{\varepsilon}
  \newcommand{\wh}{\widehat}

  \newcommand{\Q}{\mathcal{Q}}          

  \newtheorem{thm}{Theorem}[section]
  \newtheorem{lemma}[thm]{Lemma}
  \newtheorem{prop}[thm]{Proposition}

  \theoremstyle{remark}

\begin{document}

\title{A class of random Cantor measures, with applications}

\author{Pablo Shmerkin}
\address{Department of Mathematics and Statistics, Torcuato Di Tella University, and CONICET, Buenos Aires, Argentina}
\email{pshmerkin@utdt.edu}
\urladdr{http://www.utdt.edu/profesores/pshmerkin}
\thanks{P.S. was partially supported by Projects PICT 2013-1393 and PICT 2014-1480 (ANPCyT). }

\author{Ville Suomala}
\address{Department of Mathematical Sciences, University of Oulu, Finland}
\email{ville.suomala@oulu.fi}
\urladdr{http://cc.oulu.fi/~vsuomala/}
\thanks{
V.S. acknowledges support from the Centre of Excellence in
Analysis and Dynamics Research funded by the Academy of Finland}

\begin{abstract}
We survey some of our recent results on the geometry of spatially independent martingales, in a more concrete setting that allows for shorter, direct proofs, yet is general enough for several applications and contains the well-known fractal percolation measure. We study self-convolutions and Fourier decay of measures in our class, and present applications of these results to the restriction problem for fractal measures, and the connection between arithmetic structure and Fourier decay.
\end{abstract}

\maketitle

\section{Introduction}

The study of the geometric measure theoretic properties of random fractals has attracted a considerable deal of attention in the last years. At the same time, the geometry of random fractals has been investigated in many works not so much as an end in itself, but with a view on applications to problems in analysis and related fields. In \cite{ShmerkinSuomala15}, we observed that many (though certainly not all) of these works are based on a small number of key features of the underlying model, and developed a general geometric theory for an axiomatic class of random measures which include many models of interest in the previous literature, such as fractal percolation, random cascades, and Poissonian cutouts, among others. This class, which we termed \emph{spatially independent martingales}, is still too restrictive to encompass some random fractals which often arise in applications, for example cartesian powers of order $\ge 3$ of a given random measure. In a forthcoming work \cite{ShmerkinSuomala16}, we develop the theory of an even wider class of random measures which does include cartesian products of any order, and apply it to study the existence of patterns (such as angles and progressions) inside random fractals and especially fractal percolation.

In this article, we step back from the maximum generality and consider a more concrete class of random sets and measures, which is nevertheless flexible enough to obtain many of the applications that motivated our works \cite{ShmerkinSuomala15, ShmerkinSuomala16}. Bypassing the general theory provides shorter, more direct proofs of some of our results, especially those concerning self-convolutions. It also allow us to present some of the new ideas in \cite{ShmerkinSuomala16} in a concrete setting.

In Section \ref{sec:themodel} we introduce the class of random measures we will work with in the rest of the article, and establish some of their basic properties. We study self-convolutions of the random measures in Section \ref{sec:selfconvolutions}, their Fourier decay and restriction estimates in Section \ref{sec:fourierdecay}, and their arithmetic structure in Section \ref{sec:arithmeticstructure}.

\section{The model}
\label{sec:themodel}
Throughout the paper, $c,C$ denote positive and finite constants whose values may be different from line to line. Occasionally $c,C$ will be random variables which take values in $(0,+\infty)$ almost surely. Whenever necessary, we may denote the dependency of $C$ from various parameters by subscripts, for instance $C_{\varepsilon}$ denotes a positive and finite constant whose value depends on $\varepsilon>0$ but not on other parameters.

Fix $2\le M\in\N$ and for each $n\in\N$, denote by $\mathcal{Q}^d_n$ the family of all $M$-adic half open sub-cubes of the unit cube $[0,1)^d$, that is
\[
\mathcal{Q}_{n}^d=\left\{\prod_{i=1}^d[j_i M^{-n},(j_i+1) M^{-n})\,:\,0\le j_i\le M^n-1\right\}\,.
\]
For convenience, we also denote $\mathcal{Q}^1_n$ by $\mathcal{Q}_n$.

We will consider a sequence of random functions $\mu_n:[0,1)^d\to [0,+\infty)$, which in the sequel we will identify with the measures $\mu_n dx$, satisfying the following conditions for some deterministic nondecreasing sequence  $(\beta_n)_{n\in\N}$:
\begin{enumerate}
\renewcommand{\labelenumi}{(M\arabic{enumi})}
\renewcommand{\theenumi}{M\arabic{enumi}}
\item\label{mu0} $\mu_0=\mathbf{1}_{[0,1)^d}$.
\item\label{Madic} $\mu_n=\beta_n\mathbf{1}_{A_n}$, where $A_n$ is a union of cubes in $\mathcal{Q}^d_{n}$.
\item\label{martingale_condition} $\EE(\mu_{n+1}(x)\,|\,A_n)=\mu_n(x)$ for all $x\in[0,1)^d$.
\item\label{SI} Conditional on $A_n$, the random variables $\mu_{n+1}(Q)$, $Q\in\mathcal{Q}_{n+1}$ are jointly independent.
\end{enumerate}

This class of random measures is, essentially, a subclass of the subdivision martingales, which in turn fit into the more general framework of spatially independent martingale measures developed in \cite{ShmerkinSuomala15}. However, there is one direction in which our class is more general: in \cite{ShmerkinSuomala15} we required $\mu_{n+1}/\mu_n$ to be uniformly bounded. This condition can be substantially weakened in all applications, so we chose to remove it from the hypotheses. As indicated in the introduction, in order to avoid unnecessary technicalities, and to draw attention to the main ideas, we will stick the this special class in this survey.

We describe a general construction that yields sequences $(\mu_n)$ satisfying the above hypotheses. Given $Q\in\mathcal{Q}_n^d$, we denote by $\mathcal{S}(Q)$ the family of all $M^d$ cubes in $\Q_{n+1}^d$ which are contained in $Q$ (thinking of the hierarchy of $M$-adic cubes as a tree, they are the offspring of $Q$). Suppose that for each $Q\in\mathcal{Q}_n^d, n\ge 0$, there is a random subset $S_Q\subset\mathcal{S}(Q)$, such that $\{ S_Q: Q\in\mathcal{Q}_n^d, n\ge 0\}$ are independent and, for each $Q'\in\mathcal{S}(Q)$,
\[
\mathbb{P}(Q'\in S_Q) = \frac{\beta_{n}}{\beta_{n+1}}\,.
\]
Then, if we inductively set $A_0=[0,1)^d$, and
\[
A_{n+1} = \bigcup_{Q\in\mathcal{Q}_n^d, Q\subset A_n} \bigcup_{Q'\in S_Q} Q'\,,
\]
the random sequence $\mu_n =\beta_n \mathbf{1}_{A_n}$ is easily seen to satisfy \eqref{mu0}--\eqref{SI}.

The following are two important classes of examples of this construction:
\begin{enumerate}
\item[(i)] \label{ex:percolation} Fix $p\in (0,1)$ and set $\beta_n=p^{-n}$. Let $S_Q$ be the family obtained by selecting each $Q'\in\mathcal{S}(Q)$ independently with probability $p$, with all the choices independent for different $Q$. This is the well-known \emph{fractal percolation} process, and is a geometric realization of a standard Galton-Watson branching process.
\item[(ii)] \label{ex:AR} Pick $0<s<d$, and choose the sequence $\beta_n$ so that $\beta_{n+1}/\beta_n\in \{1,M^d\}$, and $C^{-1} M^{n(d-s)}\le \beta_n \le C M^{n(d-s)}$ for all $n$. We define $S_Q$ for $Q\in\mathcal{Q}_n^d$ as follows. If $n$ is such that $\beta_{n+1}=M^d \beta_n$, set $S_Q=\mathcal{S}(Q)$ (that is, all of the offspring are chosen, deterministically). Alternatively, if $\beta_{n+1}=\beta_n$, then choose $Q'$ uniformly in $\mathcal{S}(Q)$ and set $S_Q=\{Q'\}$, again with all the choices for different $Q$ independent. Note that, in this case, the number of cubes making up $A_n$ is deterministic, and moreover if $Q\in\mathcal{Q}_n^d$, $Q\subset A_n$, then $\mu_m(Q)=\beta_n^{-1}$ for all $m\ge n$. This shows that $\mu_n$ converges weakly (deterministically) to a measure $\mu$ which satisfies
    \[
    C^{-1} M^{-ns} \le \mu(Q) \le C M^{-ns}
    \]
    for all $Q\in\mathcal{Q}_n^d$ such that $Q\subset A_n$.
\end{enumerate}

\begin{figure}

   \centering
  \includegraphics[width=0.9\textwidth]{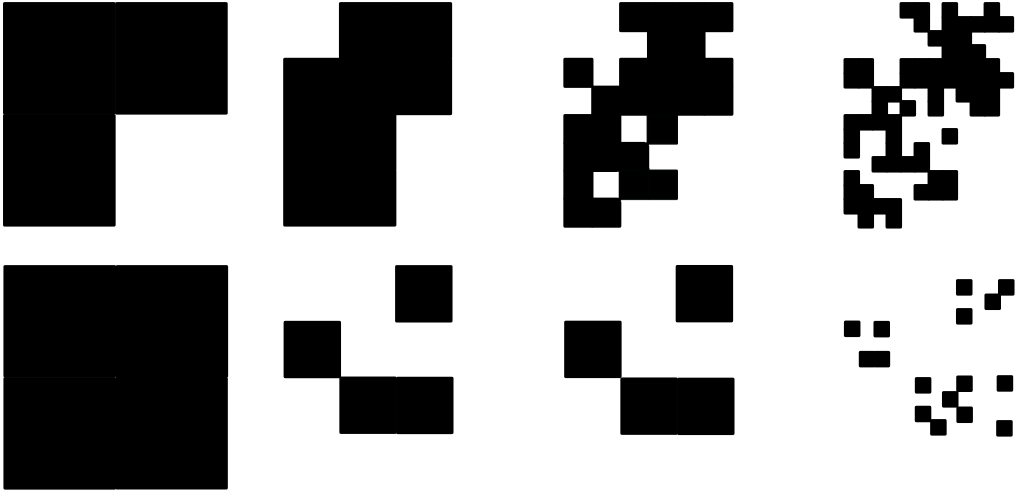}

    \caption{Illustration for the examples (i) and (ii) in the case $d=M=2$. The one on the top shows the first four steps of fractal percolation with $p=0.7$. The second one is a realization of (ii) with $\beta_1=1$, $\beta_2=\beta_3=4$ and $\beta_5=16$.}

\end{figure}

In the second example, the measures $\mu_n$ converge weakly to a non-trivial limit. In general, as in the first example, $A_n$ may be empty for some (and hence all large enough) $n$. However, provided that $\beta_n$ does not grow too quickly, the measure $\mu_n$ a.s. converge weakly to a measure $\mu$, which is non-trivial with positive probability. More precisely, denote $A=\overline{\bigcap_{n\in\N} A_n}$, and let
\begin{equation} \label{eq:def-alpha}
\underline{\alpha}=\liminf_{n\to\infty}\frac{\log_M\beta_n}{n},\quad \overline{\alpha}=\limsup_{n\to\infty}\frac{\log_M\beta_n}{n}.
\end{equation}
The following standard lemma collects the basic properties of our random model.
\begin{lemma}\label{lem:asdim}
The following holds almost surely:
\begin{enumerate}
\item The sequence of measures $\mu_n$ converges weakly to a measure $\mu$ supported on $A\subset[0,1]^d$.
\item $\overline{\dim}_B(A) \le d-\underline{\alpha}$.
\item If $\overline{\alpha}<d$, then there is a positive probability that $\|\mu\|\neq 0$, and $\dim_H A\ge d-\overline{\alpha}$ conditioned on $\|\mu\|>0$.
\item If $\underline{\alpha}=\overline{\alpha}$, then $\dim_B A=\dim_H A=\dim\mu=d-\alpha$  conditioned on $\|\mu\|>0$.
\item If $A\neq 0$ then $\|\mu\|>0$.
\end{enumerate}
\end{lemma}

Here, by $\dim\mu=s$ we mean that $\mu$ is exact dimensional of dimension $s$, that is, $\|\mu\|:=\mu(\R^d)>0$ and
\[
\dimloc(\mu,x):=\lim_{r\downarrow0}\frac{\log\mu(B(x,r))}{\log r}=s
\]
for $\mu$-almost every $x\in [0,1]^d$. Here and throughout the paper, $B(x,r)$ stands for the open ball of centre $x$ and radius $r$.

\begin{proof}[Proof of Lemma \ref{lem:asdim}]

For any continuous $f\colon [0,1]^d\to\R$, the sequence $X_n(f)=\int f\,d\mu_n$ is a non-negative martingale due to \eqref{martingale_condition}, and thus converges a.s. to a random variable $X(f)\in[0,\infty)$. Let $(f_m)_{m\in\N}$ be a uniformly dense subset of $C([0,1]^d,\R)$. Then a.s. $X(f_m)$ is well-defined for all $m$. This implies that $X(f)$ is well defined for all $f\in C([0,1]^d,\R)$ (take a subsequence of $f_m$ converging uniformly to $f$). Since $f\mapsto X(f)$ is easily checked to be a positive linear functional, it follows from the Riesz representation theorem that $\mu_n$ is almost surely weakly convergent, proving the first claim.

Let $N_n$ denote the number of cubes in $\mathcal{Q}_n^d$ forming $A_n$. Then $N_n=\beta_n^{-1}2^{Mnd}\|\mu_n\|$, so whenever $\|\mu_n\|$ is uniformly bounded (which we have seen happens almost surely), this yields $\overline{\dim}_B A\le d-\underline{\alpha}$.

The third claim follows by an application of the second-moment method. Let $\varepsilon>0$ and consider the $M$-adic metric $\kappa$ on $[0,1)^d$. That is, $\kappa(x,x)=0$ and $\kappa(x,y)=M^{-n}$ if $x$ and $y$ belong to the same element of $\mathcal{Q}_k^d$ when $k=n$ but not for any $k>n$. Let $\kappa(x,y)=M^{-m}$. Using \eqref{mu0}--\eqref{SI} gives $\PP(x,y\in A_n)=\beta^{-1}_n$ if $n\le m$ and $\PP(x,y\in A_n)=\beta_m\beta_n^{-2}$ if $n\ge m$. Since $\beta_m\le  C M^{m(\overline{\alpha}+\varepsilon)}$ for all $m$ (where $C$ depends on $\e$ and the sequence $(\beta_n)$), we infer that
\begin{equation}\label{second_moment_est}
\PP(x,y\in A_n)\le C \beta_n^{-2}|x-y|^{-\overline{\alpha}-\varepsilon}\,.
\end{equation}
Here we used the pythagorean inequality $|x-y|\le\sqrt{d}\,\kappa(x,y)$, which holds for all $x,y\in[0,1)^d$.

Let us consider the second moments of $X_n=\mu_n([0,1]^d)$. Using Fubini's theorem and \eqref{second_moment_est}, we arrive at the estimate
\begin{align}\label{eq:sm}
\EE(X_n^2)=\beta_n^2\int_{x\in[0,1)^d}\int_{y\in[0,1)^d}\PP(x,y\in A_n)\,dx\,dy\le C \iint |x-y|^{-\overline{\alpha}-\varepsilon}\,dx\, dy\,.
\end{align}
This shows that $\EE(X_n^2)$ is bounded if $\overline{\alpha}<d$, and in this case $\mu_n([0,1)^d)$ is an $L^2$-bounded martingale. By the martingale convergence theorem, $\EE(\|\mu\|)=1$ and $\EE(\|\mu\|^2)$ is bounded by a finite constant depending only on $\overline{\alpha}$ and the sequence $(\beta_n)$. By Cauchy-Schwartz,
\[
\EE(\|\mu\|)^2  =  \EE(\|\mu\|\mathbf{1}_{\|\mu\|=0})^2  \le \EE(\|\mu\|^2) \PP(\|\mu|>0),
\]
which is to say
\begin{equation} \label{eq:prob-of-survival}
\PP(A\neq\varnothing)\ge \PP(\|\mu\|>0) \ge \frac{\EE(\|\mu\|)^2}{\EE(\|\mu\|^2)} > \delta\,,
\end{equation}
where again $\delta>0$ depends only on $\overline{\alpha}$ and $(\beta_n)$.

Next, assuming that $\overline{\alpha}<d$, we estimate the $s$-energy of $\mu_n$. Calculating as in \eqref{eq:sm} we get
\begin{align*}
\EE\left(\iint|x-y|^{-s}\,d\mu_n(y)\,d\mu_n(x)\right)&=\beta_n^2\iint |x-y|^{-s}\PP(x,y\in A_n)\,dx\,dy\\
&\le C\iint |x-y|^{-s-\overline{\alpha}-\varepsilon}\,dx\,dy\,.
\end{align*}
The upper bound is independent of $n$ and it is finite whenever $s+\overline{\alpha}<d$ (and $\varepsilon>0$ is small). Using Fatou's lemma, this implies that almost surely \begin{equation}\label{eq:finite_energy}
\int_{x\in[0,1)^d}\int_{y\in[0,1)^d}|x-y|^{-s}\,d\mu(y)\,d\mu(x)<\infty
\end{equation}
whenever $s<d-\overline{\alpha}$. But it is well known that \eqref{eq:finite_energy} (for all $s<d-\overline{\alpha}$) implies that $\dim_H(A)\ge d-\overline{\alpha}$ (provided $\|\mu\|>0$) and, moreover,
\[\ldimloc(\mu,x)=\liminf_{r\downarrow0}\frac{\log\mu(B(x,r))}{\log r}\ge d-\overline{\alpha}\]
for $\mu$-almost all $x\in[0,1)^d$.

In the case $\overline{\alpha}=\underline{\alpha}=\alpha$, recalling that $\dim_H\le\dim_B A\le d-\underline{\alpha}$, we conclude that a.s. $\udimloc(\mu,x)=\limsup_{r\downarrow0}\frac{\log(\mu(B(x,r))}{\log r}= d-\alpha$ for $\mu$-almost all $x$ and in particular, $\dim_H A=\dim\mu=d-\alpha$ almost surely on $A\neq\varnothing$.

It remains to show that $A$ and $\mu$ have the same survival probability. We note that for the example (ii) above it holds (deterministically) that $\|\mu\|=1$, whereas for the example (i) the fact that $A$ and $\mu$ have the same survival probability follows from a standard zero-one law for Galton-Watson branching processes (see \cite[Corollary 5.7]{LyonsPeresbook}). The general case still carries a weak form of self-similarity allowing to deduce that $\|\mu\|>0$ a.s. on $A\neq\varnothing$. A key point is the existence of an increasing sequence $n_k$ and $\delta>0$ such that for all $Q\in\mathcal{Q}_{n_k}^d$, $\PP(\mu(Q)>0\,|\,Q\in A_{n_k})>\delta$. We omit the technical details.
\end{proof}

\section{self convolutions}
\label{sec:selfconvolutions}

Recall that if $\mu,\nu$ are two finite Borel measures on $\R^d$, then their convolution is the push-down of $\mu\times \nu$ under the addition map $(x,y)\mapsto x+y$, that is,
\[
\int f\,d(\mu*\nu) = \int f(x+y)d\mu(x)d\nu(y)
\]
for all bounded functions $f:\R^d\to\R$.

In this section we investigate self-convolutions of the random measures $\mu$. It has been known for a long time that the self-convolution $\mu*\mu$ of a singular measure on $\R$ can be absolutely continuous, and the density can even be H\"{o}lder continuous. K\"{o}rner \cite{Korner08} obtained an optimal estimate in terms of the dimension of (the support of) $\mu$ and the H\"{o}lder exponent of the self-convolution on $\R$. Recently, Chen and Seeger \cite{ChenSeeger15}, by adapting K\"{o}rner's construction, extended this to $\R^d$ and self-convolutions of any order. The constructions of K\"{o}rner, and of Chen and Seeger are ad-hoc and in some sense geometrically irregular, for example they are never Ahlfors-regular. In this section we show that a similar result holds for the measures in our class, including fractal percolation. The result may be equivalently stated in terms of orthogonal projections: denoting by $\Pi\colon(\R^d)^m\to\R^d, (x_1,\ldots,x_m)\to\sum_{i=1}^m x_i$ the orthogonal projection onto the $d$-dimensional diagonal of $(\R^d)^m$, the result says that the push-forward $\mu^n\circ\Pi^{-1}$ is absolutely continuous with a H\"older density, whenever $\overline{\alpha}<(m-1)d/m$ (where $\overline{\alpha}$ is defined in \eqref{eq:def-alpha}). For fractal percolation sets $A$ in $\R^2$, the study of orthogonal projections was initiated by Rams and Simon \cite{RamsSimon14, RamsSimon15} and for the planar fractal percolation measure with almost sure dimension $>1$ the H\"older continuity of the projections was first addressed by Peres and Rams \cite{PeresRams14}. Rams and Simon \cite{RamsSimon14} studied the dimension of the sums of $m$ independent copies of the fractal percolation set for any $m\ge 2$; their methods do not work for convolutions of the natural measure. See also \cite{SimonVago14, ShmerkinSuomala15b} for other related recent work. As will be explained below, the study of self-convolutions, especially of order $\ge 3$, requires new ideas.

\begin{thm} \label{thm:self-convolutions}
Suppose $(\mu_n)$ satisfy \eqref{mu0}--\eqref{SI}, and let $\overline{\alpha}$ be as in \eqref{eq:def-alpha}. If $\overline{\alpha}<d/2$ and
\begin{equation}\label{eq:beta_n_bound}
\lim_{n\to\infty} \log\beta_{n+1}/\log\beta_n=1,
\end{equation}
then, conditioned on $\|\mu\|>0$, the convolution $\mu*\mu$ is almost surely absolutely continuous, and the density is H\"older continuous with exponent $\gamma$ for any $\gamma<d/2-\overline{\alpha}$.

If $\overline{\alpha}<2d/3$ the same holds for $\mu*\mu*\mu$ and H\"older exponents
\begin{equation}
\gamma<
\begin{cases}
d-\tfrac32\overline{\alpha}\text{ when }d/2\le\overline{\alpha}<2d/3\\
\frac{1}{2}(d-\overline{\alpha})\text{ when }0<\overline{\alpha}\le d/2
\end{cases}
\end{equation}
Likewise, if $\overline{\alpha}<(m-1)d/m$ for some $3<m\in\N$, then the $m$-fold self convolution $\mu^{*m}$ is a.s. absolutely continuous and the density is H\"older with a quantitative exponent $\gamma=\gamma(d,m,\overline{\alpha})$.
\end{thm}
We make some remarks on the statement. The hypotheses hold, in particular, when $\log\tfrac1n \beta_n\to \alpha\in (0,d/2)$, and more concretely for our classes of examples (i) and (ii). At least for double convolutions, the range of H\"{o}lder exponents is optimal (up to the critical exponent), in the sense that if $\mu$ is any measure supported on a set of Hausdorff dimension $d-\alpha$, then $\mu*\mu$ cannot have a H\"{o}lder density of exponent larger than $d/2-\alpha$, see \cite{Korner08, ChenSeeger15}.

For clarity of exposition, we only present the proof in the case of dimension $d=1$ and double and triple convolutions; these cases already contain the main ideas of the general case, while being technically much simpler. The proof of the general case can be found in \cite{ShmerkinSuomala15,ShmerkinSuomala16}. The proof has a deterministic and a random component; in order to clarify the ideas, we deal with them separately starting with the deterministic result. We further split this into the cases of $\mu*\mu$ and $\mu*\mu*\mu$.

\begin{prop} \label{prop:Holder-cont-deterministic-step}
Let $(\mu_n)$ be a sequence of measures in $[0,1)^d$ satisfying \eqref{mu0} and \eqref{Madic} and suppose that $\mu_n$ weakly converges to a non-trivial measure $\mu$. For each $u\in[0,2]^d$, define
\[Y_{n}^u=\int_{\Pi^{-1}(u)}\mu_n\times\mu_n\,d\mathcal{H}^d\,,\]
where $\Pi\colon\R^{2d}\to\R^d$, $(x,y)\mapsto x+y$

Let $0<\widetilde{\gamma}\le 1$ and let $\Gamma_n\subset[0,2]^d$ be $\delta_n$-dense for each $n\in\N$, where $\delta_n=M^{-n(\widetilde{\gamma}+d)}\beta_{n+1}^{-2}$. Suppose that for some $C<\infty$,
\begin{equation}\label{eq:speed_assumption}
|Y_{n+1}^u-Y_{n}^u|\le C M^{n\widetilde{\gamma}}(1+\sqrt{Y_n^u})
\end{equation}
for each $n\in\N$ and each $u\in\Gamma_n$. Then $\mu*\mu$ is absolutely continuous and its density is H\"older of any exponent $\gamma<\widetilde{\gamma}$.
\end{prop}

\begin{prop} \label{prop:Holder-cont-deterministic-step-3}
Let $(\mu_n)$ be a sequence of measures in $[0,1)^d$ satisfying \eqref{mu0} and \eqref{Madic} and suppose that $\mu_n$ weakly converges to a non-trivial measure $\mu$. For each $u\in[0,3]^d$, define
\[Y_{n}^u=\int_{\Pi^{-1}(u)}\mu_n\times\mu_n\times\mu_n\,d\mathcal{H}^{2d}\,,\]
where $\Pi\colon\R^{3d}\to\R^d$, $(x,y,z)\mapsto x+y+z$.

Let $\widetilde{\gamma}\le1$ and $\Gamma_n\subset[0,3]^d$ be $\delta_n$-dense for each $n\in\N$, where $\delta_n=M^{-n(\widetilde{\gamma}+2d)}\beta_{n+1}^{-3}$. Suppose that for some $C<\infty$,
\begin{equation*}
|Y_{n+1}^u-Y_{n}^u|\le C M^{n\widetilde{\gamma}}(1+\sqrt{Y_n^u})
\end{equation*}
for each $n\in\N$ and each $u\in\Gamma_n$. Then $\mu*\mu*\mu$ is absolutely continuous and its density is H\"older with any exponent $\gamma<\widetilde{\gamma}$.
\end{prop}

We will only give a proof for Proposition \ref{prop:Holder-cont-deterministic-step}; Proposition \ref{prop:Holder-cont-deterministic-step-3} may be proved with minor modifications.

\begin{proof}[Proof of Proposition \ref{prop:Holder-cont-deterministic-step} for $d=1$]
Recall that $\mu*\mu$ is the image of $\mu\times\mu$ under the addition map $\Pi\colon(x,y)\mapsto x+y$, $\R^2\to\R$. We identify $\Pi$ with the orthogonal projection onto the diagonal line $\{x=y\}\subset\R^2$ multiplied by a factor $\sqrt{2}$. Thus, for all $x\in\R$ and $r>0$, Fubini's theorem implies
\begin{equation}\label{eq:coarea}
\begin{split}
\mu*\mu(B(x,r))&=(\mu\times\mu)(\Pi^{-1}B(x,r))\\
&\le \liminf_{n\to\infty} \mu_n\times\mu_n(\Pi^{-1}(B(x,r)))\\
&=   \liminf_{n\to\infty} \beta_n^2\leb^{2}(\Pi^{-1}(B(x,r))\cap (A_n\times A_n)) \\
&=\frac{1}{\sqrt{2}}\liminf_{n\to\infty} \int_{B(x,r)} \beta_n^2  \mathcal{H}^1( \Pi^{-1}(u)\cap  (A_n\times A_n)) du\\
&=\frac{1}{\sqrt{2}}\liminf_{n\to\infty} \int_{B(x,r)} Y_n^{u} \, du\,.
\end{split}
\end{equation}

To complete the proof, it is enough to show that $Y^u:=\lim_{n\to\infty}\int_{\Pi^{-1}(u)}\mu_n\times\mu_n\,d\mathcal{H}^1$ is well defined for all $u\in\R$, and that $u\mapsto Y^u$ is H\"older continuous. Indeed, once this is verified, it follows that $\sup_{n\in\N, u\in[0,2]}Y_{n}^u$ is bounded, and \eqref{eq:coarea} yields $\mu*\mu(B(x,r))\le C r$ for all $x\in\R$. Moreover, replacing the open balls in \eqref{eq:coarea} with closed balls, it follows that \eqref{eq:coarea} is actually an equality. Furthermore, the density $d\mu*\mu(x)$ equals $Y^{x}/\sqrt{2}$. In particular, $Y^x$ is not zero for all $x$, since $\mu$ is non-trivial.

The proof of the H\"older continuity of $u\mapsto Y^u$ relies on the following modulus of continuity for $u\mapsto Y_{n}^u$:
\begin{equation}\label{eq:aprioriHolder}
|Y_{n}^u-Y_{n}^{u'}|\le 3\beta_n^2 M^{n}|u-u'|\,.
\end{equation}
This follows by elementary geometry. Indeed, for each fixed $Q\in\mathcal{Q}_n^2$, the map $u\mapsto\mathcal{H}^1(\Pi^{-1}(u)\cap Q)$ is Lipschitz continuous with Lipschitz constant $\sqrt{2}$. Since each $\Pi^{-1}(u)$ intersects at most $2 M^{n}$ such cubes, the estimate \eqref{eq:aprioriHolder} follows.

Since $\Gamma_n$ is $\delta_n$-dense with $\delta_n=M^{-n(\widetilde{\gamma}+1)}\beta_{n+1}^{-2}$, Equation \eqref{eq:aprioriHolder} implies that for each $u\in[0,2]$, there is $u'\in\Gamma_n$ such that $|Y^u_{n+1}-Y^{u'}_{n+1}|\le 3\beta_{n+1}^2 M^{n+1}|u-u'|\le 3 M^{1-n\widetilde{\gamma}}$, and likewise for $|Y^u_n-Y^{u'}_n|$.

Let $X_n=1+\sup_{u\in[0,2]}Y_{n}^u$. Combining the above with \eqref{eq:speed_assumption}, and using the triangle inequality for $u$, we get
\begin{equation}\label{eq:X_n}
\begin{split}
|X_{n+1}-X_n|&\le\sup_{u\in [0,2]}|Y_{n+1}^u-Y_n^u|\\
&\le 6\beta_{n+1}^2M^{n+1}\delta_n+\max_{u\in\Gamma_n}|Y_{n+1}^u-Y_n^u|\\
&\le CM^{-n\widetilde{\gamma}}X_n\,,
\end{split}
\end{equation}
for all $n\in\N$. This implies that $X_{n}$ converges to a finite limit $X$, so in particular $\overline{X}=\sup_{n}X_n$ is finite, and that (absorbing $\overline{X}$ into the constant $C$)
\begin{align*}
|X_n-X|\le\sum_{k\ge n}|X_{k+1}-X_k|\le C M^{-n\widetilde{\gamma}}
\end{align*}
holds for all $n\in\N$. Furthermore, the same is true for $Y_{n}^u$, that is, the limit $Y^u=\lim_n Y_{n}^u$ exists for all $u\in\R$, and
\begin{equation}\label{eq:speed_Yn}
|Y_n^u-Y^u|\le C M^{-n\widetilde{\gamma}}\,,
\end{equation}
whenever $n\in\N$ and $u\in[0,2]$.

Define $\gamma_0=0$ and  $\gamma_{m+1}=\widetilde{\gamma}/(1+\widetilde{\gamma}-\gamma_m)$. We will show that $u\mapsto Y^u$ is H\"older with exponent $\gamma_m$ for all  $m\in\N$. Since $\lim_{m\to\infty}\gamma_m=\widetilde{\gamma}$,
we thus recover all the H\"older exponents $\gamma<\widetilde{\gamma}$. We proceed by induction on $m$. Since $\overline{X}$ is bounded, we have a uniform upper bound for $Y^{u}$ (so we may say that $u\mapsto Y^u$ is H\"older with exponent $\gamma_0=0$). Suppose that $u\mapsto Y^u$ is $\gamma_m$-H\"older, i.e. $|Y^u-Y^{u'}|\le C|u-u'|^{\gamma_m}$ for all $u,u'\in[0,2]$. Let $Q=[a,b]\in\mathcal{Q}_n$. Recalling \eqref{eq:speed_Yn} and using the triangle inequality,
\begin{equation}\label{eq:triangle}
\begin{split}
|Y_n^b-Y_n^{a}|&\le |Y^b_n-Y^b|+|Y^b-Y^a|+|Y^a-Y^a_n|\\
&\le C M^{-n\widetilde{\gamma}}+CM^{-n\gamma_m}
\le C M^{n(1-\gamma_m)}M^{-n}\,.
\end{split}
\end{equation}
(Note that $\gamma_m<\widetilde{\gamma}$ for all $m$.) Since each preimage $\Pi^{-1}(Q)$, $Q\in\mathcal{Q}_n$, consists of halves of cubes in $\mathcal{Q}_n^2$, either above or below the top-left to bottom-right diagonal, we see that the map $u\mapsto Y^u_n$ is linear on each $Q\in\mathcal{Q}_n$ (and this holds also for the $M$-adic subcubes of $[1,2)$). Hence we arrive at the following key estimate:
\begin{equation}\label{eq:Lipschitz}
|Y_n^u-Y_n^{u'}|\le C M^{n(1-\gamma_m)}|u-u'|\text{ whenever }|u-u'|\le M^{-n}\,.
\end{equation}
Note that this improves upon the crude estimate \eqref{eq:aprioriHolder}. Now let $u,u'\in[0,2]$ be arbitrary and let $n\in\N$ such that $M^{-n(1+\widetilde{\gamma}-\gamma_m)}<|u-u'|\le M^{-(n-1)(1+\widetilde{\gamma}-\gamma_m)}$. Then \eqref{eq:speed_Yn} and \eqref{eq:Lipschitz} yield
\begin{align*}
|Y^u-Y^{u'}|&\le |Y^u-Y^{u}_n|+|Y^u_{n}-Y^{u'}_n|+|Y^{u'}_n-Y^{u'}|\\
&\le C M^{-n\widetilde{\gamma}}+C M^{n(1-\gamma_m)}|u-u'|\,.
\end{align*}
Recalling the choice of $n$, both terms are bounded by $C|u-u'|^{\gamma_{m+1}}$. This confirms that $u\mapsto Y^u$ is H\"older continuous with exponent $\gamma_{m+1}$, finishing the proof.
\end{proof}

In order to establish Theorem \ref{thm:self-convolutions}, it then remains to check that the self-convolutions satisfy the assumptions of Proposition \ref{prop:Holder-cont-deterministic-step}. For this, we will use a generalization of Hoeffding's inequality due to Janson \cite{Janson04}, which allows for some dependencies among the random variables. Recall that a graph with a vertex set $I$ is a \emph{dependency graph} for the random variables $\{X_i\,:\,i\in I\}$ if whenever $i\in I$ and $J\subset I$ are such that there is no edge between $i$ and any element of $J$, the random variable $X_i$ is independent of $\{X_j\,:\,j\in J\}$.
\begin{lemma} \label{lem:HoeffdingJanson}
Let $\{ X_i: i\in I\}$ be zero mean random variables uniformly bounded by $R>0$, and suppose there is a dependency graph with degree $\Delta$. Then
\begin{equation} \label{eq:Hoeffding-Janson}
\mathbb{P}\left(\left|\sum_{i\in I} X_i\right|> \varrho\right) \le 2\exp\left(\frac{- 2 \varrho^2}{(\Delta+1)|I| R^2}\right).
\end{equation}
\end{lemma}

The case of double convolutions is substantially simpler than the triple ones, so we present its proof first.
\begin{proof}[Proof of Theorem \ref{thm:self-convolutions} for $\mu*\mu$ in the case $d=1$]

For $u\in[0,2]$ and $n\in\N$, let $Y_n^u$ be defined as in Proposition \ref{prop:Holder-cont-deterministic-step}.
For each fixed $u\in[0,2]$, we will consider several subfamilies of $\mathcal{Q}_n\times\mathcal{Q}_n$ as follows: Let $Q'$ denote the unique diagonal cube $Q'=Q\times Q\in\mathcal{Q}_n\times\mathcal{Q}_n$ which intersects $\Pi^{-1}(u)$.

Further, for each $n\le\ell\in\N$, set
\[
\widetilde{\mathcal{Q}}^\ell=\{Q\in\mathcal{Q}_n^2\setminus\{Q'\}\,:\,Q\subset A_n\times A_n, M^{-(\ell+1)}\le\mathcal{H}^1(\Pi^{-1}(u)\cap Q)\leq M^{-\ell}\}\,.
\]
Given $Q\in\Q_n^2$, let
\[
X_Q  = \int_{Q\cap \Pi^{-1}(u)} (\mu_{n+1}\times\mu_{n+1}-\mu_n\times\mu_n) \, d\mathcal{H}^1\,.
\]
For each $\ell$ we define a graph on the vertices $\widetilde{\mathcal{Q}}^\ell$ as follows: there is an edge between $(I_1\times J_1)$ and $(I_2\times J_2)$ (where $I_i,J_i\in\Q_n$) if and only if $\{ I_1, J_1\}\cap \{ I_2,J_2\}\neq\varnothing$. It follows from \eqref{SI} that this is a dependency graph for $(X_Q|A_n)_{Q\in \widetilde{\mathcal{Q}}^\ell}$. Elementary geometry shows that the degree of this graph is at most $8$: given $I\in\Q_n$, there are at most two intervals $J\in\Q_n$ such that $u\in \Pi(I\times J)$ and at most two intervals $J'\in\Q_n$ such that $u\in \Pi(J'\times I)$.

On the other hand, $\EE(X_Q|A_n)=0$ for all $Q\in\widetilde{\mathcal{Q}}^\ell$ by \eqref{martingale_condition}. (The reason we exclude $Q'$ is that $\EE(X_{Q'}|A_n)\neq 0$.) Furthermore,
\[
|X_Q|\le \sqrt{2}\beta^2_{n+1} M^{-\ell}
\]
for all $Q\in\widetilde{\mathcal{Q}}^{\ell}$ and
\[
\#\widetilde{\mathcal{Q}}^{\ell}\le M^{\ell+1}\beta_n^{-2} Y_n^{u}\,,
\]
since $Y^{u}_n\ge\sum_{Q\in\widetilde{\mathcal{Q}}^{\ell}}\beta_n^2\,\mathcal{H}^1(Q\cap \Pi^{-1}(u))$. Applying Lemma \ref{lem:HoeffdingJanson}, we get
\begin{equation}\label{eq:Hoefdinng_kappa_l}
\PP\left(\left|\sum_{Q\in\widetilde{\mathcal{Q}}^\ell}X_Q\right|>\kappa_\ell\sqrt{Y_{n}^u}\right)\le 2\exp\left(\frac{-c\beta^2_n\kappa_\ell^2 M^{\ell-1}}{\beta_{n+1}^4}\right)
\end{equation}
for any $\kappa_\ell>0$. To be more precise, this bound holds conditional on $A_n$, but as the upper bound does not actually depend on $A_n$, it also holds unconditionally.

Pick $\tau>0$ such that $0<2\tau<1/2-\overline{\alpha}$ and define
\[\varepsilon_n=M^{n(\overline{\alpha}+\tau-1/2)}\,,\]
and
\[\kappa_\ell=\frac{\varepsilon_n}{2(\ell-n+2)^2}\,.\]
We note the bounds
\[
\sum_{n=\ell-1}^\infty \kappa_\ell \le 1,
\]
\begin{equation}\label{eq:diagonalQ_est}
X_{Q'}\le \sqrt{2}\beta_{n+1}^2M^{-n}\le C'\e_n,
\end{equation}
\[
M^{\ell-n}(\ell-n+2)^{-4}  \ge c(\ell-n+2)\text{ for }\ell> n-2,\\
\]
\[
\beta_n \le C M^{(\overline{\alpha}+\tau/4)n},
\]
\[
\beta_{n+1}\le C\beta_n M^{n\tau/4},
\]
where the last one follows from \eqref{eq:beta_n_bound}. Using this, summing up over all $\ell> n-2$, and noting that $\widetilde{\mathcal{Q}}^{\ell}=\varnothing$ for $\ell\le n-2$, we arrive at
\begin{equation}\label{eq:corollary_of_Hoeffding}
\begin{split}
&\PP\left(|Y_{n+1}^u-Y_n^u|>C'\e_n+\sqrt{Y_n^u}\varepsilon_n\right)\\
&\le\PP\left(\left|\sum_{Q\in\widetilde{\mathcal{Q}}^\ell}X_Q\right|>\kappa_\ell\sqrt{Y_{n}^u}\text{ for some }\ell\right)\\
&\le\sum_{\ell\ge n-1}2\exp\left(\frac{-c\beta^2_n\kappa_\ell^2 M^{\ell-1}}{\beta_{n+1}^4}\right)\\
&\le\sum_{\ell\ge n-1}2\exp\left(-c M^{-(2\overline{\alpha}+\tau/2)n} M^{-\tau n/2} M^{2(\overline{\alpha}+\tau-1/2)n} M^{n}M^{\ell-n+2}(\ell-n+2)^{-4}   \right)\\
&\le\sum_{\ell\ge n-1}2\exp\left(-c M^{\tau n} (\ell-n+2)  \right)\\
&\le C\exp(-c M^{\tau n})\,.
\end{split}
\end{equation}

We emphasize that the constants $c,C,C'$ are independent of $n, u, \ell$ and that $C'$ is the constant from \eqref{eq:diagonalQ_est}. Now let $\Gamma_n\subset[0,2]$ be a $\delta_n$-dense set, with $\#\Gamma_n\le 2 \delta_n^{-1}$ for $\delta_n=\beta_{n+1}^{-2} M^{-n(\overline{\alpha}-\tau-3/2)}$.
From \eqref{eq:corollary_of_Hoeffding}, we derive the following estimate
\begin{align*}
&\PP\left(|Y_{n+1}^u-Y_{n}^u|>C'\e_n +\sqrt{Y_{n}^u}\,\e_n\text{ for some }u\in\Gamma_n\right)<C\delta_n^{-1}\exp\left(-c M^{\tau n}\right)\\
&\le C M^{n(3\overline{\alpha}+2\tau+3/2)}\exp(-c M^{\tau n})\le C\exp(-M^{\tau n/2})\,,
\end{align*}
for some $C<\infty$ independent of $n$. The Borel-Cantelli lemma yields a random $N_0\in\N$ such that
\[
|Y_{n+1}^u-Y_{n}^u|\le C M^{n(\overline{\alpha}+\tau-1/2)}(1+\sqrt{Y_n^u})\quad\text{for all }n\ge N_0, u\in\Gamma_n\,.
\]
Making $C$ larger if necessary (depending on $N_0$), this holds for all $n\in\N$. Thus we have verified the assumptions of Proposition \ref{prop:Holder-cont-deterministic-step} for $\widetilde{\gamma}=1/2-\overline{\alpha}-\tau$. Since $\tau>0$ is arbitrarily small, this finishes the proof.
\end{proof}

We turn to the case of triple convolutions.

\begin{proof}[Proof of Theorem \ref{thm:self-convolutions} for $\mu*\mu*\mu$ in the case $d=1$]
We will give a detailed proof for $\overline{\alpha}\in (1/2,2/3)$ and briefly explain the required changes for $0<\overline{\alpha}\le 1/2$ at the end. We will skip the details of calculations that are routine or similar to chose carried our in the course of the proof of the double convolution case.

For $u\in[0,3]$, let $Y_n^u$ and $\Pi$ be defined as in the Proposition \ref{prop:Holder-cont-deterministic-step-3}.
The main reason why we cannot directly apply the same argument as for $\mu*\mu$ is that independence and the martingale condition break down in a much more severe way. For instance, if we let
\begin{equation}\label{eq:X_Q}
X_Q=\int_{Q\cap \Pi^{-1}(u)}(\mu_{n+1}\times\mu_{n+1}\times\mu_{n+1}-\mu_n\times\mu_n\times\mu_n)\,d\mathcal{H}^2\,,\end{equation}
the random variables $(X_{Q_i})$ are not independent conditional on $A_n$, whenever $Q_i\in\mathcal{Q}^3_n$ have the same projection onto one of the coordinate axes. Since $\Pi^{-1}(u)$ intersects planes of the form $\{x=c\}$, $\{y=c\}$, $\{z=c\}$ along a line, there could be many such dependent cubes contained in $A_{n}^3$. Regarding the failure of the martingale condition, note that if $Q\in\mathcal{Q}^3_n$ has two common coordinate projections (e.g. $Q=Q_1\times Q_1\times Q_2$ for some $Q_{1},Q_2\in\mathcal{Q}_n$), and if $Q\subset A_{n}^3$, then $\EE(X_Q\,|\,A_n)\neq 0$. Again, unlike in the planar case, there can be many such cubes $Q\subset A_{n}^3$ along the semidiagonals $\{x=y\}$, $\{x=z\}$, $\{y=z\}$.

To overcome these issues, we will do a joint probabilistic induction in $n$ for our main quantity of interest $Y_n^u$ and for related, two-dimensional quantities (involving the marginals $\mu_n\times\mu_n$) that will allow us to find dependency graphs for the $X_Q$ with suitably small degrees, so that ultimately the general scheme in the proof of the double convolution case can be pushed through. We note that for higher order convolutions, a similar argument still works but involves an even more complicated induction also in the order of the convolution.

Let us denote by $\mathcal{R}$ the family of lines in $\R^2$ of the form
\begin{align*}
x+y&=u''\text{ or }\\
2x+y&=u''\text{ or}\\
x+2y&=u''\,,
\end{align*}
for some $u''\in[0,3]$. For $V\in\mathcal{R}$, define
\[
\widetilde{Y}^{V}_n=\int_V\mu_n\times\mu_n\,d\mathcal{H}^1\,.
\]
Since $\overline{\alpha}>1/2$, the random variables $\widetilde{Y}^{V}_n$ will no longer be uniformly bounded. However, using ideas similar to the case $\overline{\alpha}<1/2$, we will be able to derive a rather sharp growth estimate for $\sup_{V\in\mathcal{R}}\widetilde{Y}^V_n$. To this end, define
\[
\widetilde{X}_n=1+\sup_{V\in\mathcal{R}}\widetilde{Y}^V_n\,.
\]
Let $\widetilde{\gamma},\varepsilon>0$ be such that
\begin{equation}\label{eq:widgamma_def}
2-3\overline{\alpha}-2\widetilde{\gamma}-4\varepsilon>0.
\end{equation}
We consider a very large parameter $L<\infty$ (which remains fixed for now, but will tend to $+\infty$ later on), and claim that
\begin{equation}\label{eq:to_bound_the_dep_degree}
\PP(\widetilde{X}_{n+1}>L M^{(n+1)(2\overline{\alpha}-1+2\varepsilon)}\,|\,\widetilde{X}_n\le L M^{n(2\overline{\alpha}-1+2\varepsilon)})\le C\exp(-c M^{n\varepsilon/2})
\end{equation}
for some $0<c,C<\infty$ independent of $L$ and $n$. In order not to interrupt the flow of the proof, this is proved in Lemma \ref{lem:technical-estimate}
below.

Now, let us return to the random variables $Y_n^u$. Write $Y_{n+1}^u-Y_n^u=\sum_{Q\in\mathcal{Q}_{u,n}}X_Q$, where $X_Q$ is defined as in \eqref{eq:X_Q} and $\mathcal{Q}_{u,n}$ consists of those $Q\in\mathcal{Q}^3_n$ for which $A_n^3\cap Q\cap\Pi^{-1}(u)\neq\varnothing$.  We claim that for each $u\in [0,3]$ there is a dependency graph for $(X_Q|A_n: Q\in \mathcal{Q}_{u,n})$ of degree at most
\begin{equation}\label{eq:dep_degree}
\Delta(n)=C\widetilde{X}_n\beta_{n}^{-2}M^{n},
\end{equation}
where $C$ is independent of $u$ and $n$. This will be proved in Lemma \ref{lem:dependency-degree-bound} below. Let $\mathcal{Q}'_{u,n}\subset \mathcal{Q}_{u,n}$ be the family of semidiagonal cubes, that is, cubes intersecting one of the semi-diagonals $\{x=y\}$, $\{x=z\}$, $\{y=z\}$. In Lemma \ref{lem:dependency-degree-bound} we will also show that $\#\mathcal{Q}'_{u,n}\le \Delta(n)$. In particular,
\[
\sum_{Q\in\mathcal{Q}'_{u,n}}X_Q \le C\#\mathcal{Q}'_{u,n}\beta_{n+1}^3M^{-2n} \le C \widetilde{X}_n \beta_n^{-2}\beta_{n+1}^3 M^{-n},
\]
so we can estimate
\begin{equation} \label{eq:estimate-diagonal-cubes}
|Y_{n+1}^u-Y_{n}^u|\le C \widetilde{X}_n \beta_n^{-2}\beta_{n+1}^3 M^{-n}+\sum_{Q\in\mathcal{Q}_{u,n}\setminus\mathcal{Q}'_{u,n}} X_Q\,.
\end{equation}

Note that \eqref{martingale_condition} implies $E(X_Q|A_n)=0$ for all $Q\in\Q_{u,n}\setminus \Q'_{u,n}$. Thus, after conditioning on $A_n$, we can decompose $\Q_{u,n}\setminus \Q'_{u,n}$ into families $\Q_{u,n,\ell}$ according to the area of $Q\cap \Pi^{-1}(u)$, similar to the proof in the case of double-convolutions. Applying Lemma \ref{lem:HoeffdingJanson} for each of these families, with $\Delta=\Delta(n)$ given in \eqref{eq:dep_degree},  together with \eqref{eq:estimate-diagonal-cubes}, a calculation similar to \eqref{eq:Hoefdinng_kappa_l}--\eqref{eq:corollary_of_Hoeffding} yields the following estimate for each $u\in[0,3]$:
\begin{align*}
\PP&\left(|Y^u_{n+1}-Y^u_n|>C L \beta_{n+1}^3\beta_n^{-2} M^{n(2\overline{\alpha}-2+2\varepsilon)}+L M^{-n\widetilde{\gamma}}\sqrt{Y^{u}_n}\,|\,\widetilde{X}_n\le L M^{n(2\overline{\alpha}-1+2\varepsilon)}\right)\\
&\le 2\exp\left(-c L \beta_{n}^5\beta_{n+1}^{-6}M^{n(2-2\overline{\alpha}-2\widetilde{\gamma}-2\varepsilon)}\right)\,,
\end{align*}
for some constants $0<c,C<\infty$ independent of $n$, $L$ and $u$. Note that here $\beta_{n+1}^3\beta_n^{-2} M^{n(2\overline{\alpha}-2+3\varepsilon)}\le C M^{-2\widetilde{\gamma}n}$ due to \eqref{eq:beta_n_bound}, \eqref{eq:widgamma_def} and the definition of $\overline{\alpha}$ and furthermore,
\[
\beta_{n}^5\beta_{n+1}^{-6}M^{n(2-2\overline{\alpha}-2\widetilde{\gamma}-2\varepsilon)}\ge c M^{n(2-3\overline{\alpha}-2\widetilde{\gamma}-3\varepsilon)}\ge c M^{n\varepsilon}
\]
for some $0<c,C<\infty$ independent of $n$.

Let $\Gamma_n\subset[0,3]$ be $\delta_n$-dense with $\#\Gamma_n\le C\delta_{n}^{-1}$ where $\delta_n=\beta_{n+1}^{-3}M^{(-2-\widetilde{\gamma})n}$. Applying the above
estimate for each $u\in\Gamma_n$ yields
\begin{align}
\notag
\PP&\left(|Y^u_{n+1}-Y^u_{n}|>C L M^{-n\widetilde{\gamma}} (1+\sqrt{Y_n^u})\text{ for some }u\in\Gamma_n\,|\,\widetilde{X}_n\le L M^{n(2\overline{\alpha}-1+2\varepsilon)}\right)\\
&\le C \delta^{-1}_{n}\exp(-c L M^{\varepsilon n})\le C\exp(-LM^{\varepsilon n/2})\,.
\label{eq:denseGn}
\end{align}

Now let $\mathcal{A}_{n,L}$ denote the event that $\widetilde{X}_{n+1}\le L M^{n(2\overline{\alpha}-1+2\varepsilon)}$ and $|Y_{n+1}^u-Y_{n}^u|\le L M^{-n\widetilde{\gamma}}(1+\sqrt{Y_n^u})$ for all $u\in\Gamma_n$. Combining \eqref{eq:to_bound_the_dep_degree} and \eqref{eq:denseGn} gives $\PP(\mathcal{A}_{n+1,L}|\mathcal{A}_{n,L})\ge 1-C\exp(-cM^{n\varepsilon/2})$, so that
\[
\PP(\mathcal{A}_{n,L} \text{ holds for all } n\ge N) \longrightarrow 1
\]
as $N\to\infty$, uniformly in $L$. On the other hand, for each $N\in\N$, the events $\mathcal{A}_{n,L}$ for $n\le N$ hold deterministically, provided $L=L_{N}$ is chosen large enough. Combining these facts, we conclude that a.s. there is $L<\infty$ such that $\mathcal{A}_{n,L}$ holds for all $n\in\N$. In particular,
\[
|Y_{n+1}^u-Y_n^u|\le L M^{-n\widetilde{\gamma}}(1+\sqrt{Y_n^u})
\]
for all $n\in\N$, $u\in\Gamma_n$.

Recalling \eqref{eq:widgamma_def}, we see that $\widetilde{\gamma}$ may be chosen to be arbitrarily close to $1-\tfrac32\overline{\alpha}$, so we recover all the desired H\"older exponents from Proposition \ref{prop:Holder-cont-deterministic-step-3}.

Finally, let us briefly discuss the situation for $\overline{\alpha}<1/2$. Note that the H\"older continuity (with exponents $0<\gamma<1/2-\overline{\alpha}$) follows since $\mu*\mu$ is a.s H\"older continuous and adding one more convolution cannot decrease the H\"older exponent. It is possible to improve the H\"older exponent to the range $\gamma<\tfrac{\overline{\alpha}}{2}-\tfrac12$ by modifying the above argument in the case $\tfrac12\le \overline{\alpha}<\tfrac32$. Inspecting the proof,  we see that we obtained for all $\varepsilon>0$ an a.s. upper bound of order $M^{n\varepsilon}$ for the degree of the dependency graph of $\mathcal{Q}_{u,n}$. When $\overline{\alpha}<1/2$, the degree of the dependency graph will be much larger, but an efficient bound can still be given by following the ideas in the proof of the double convolution case. Tracking the numerical values one checks that the assumptions of Proposition \ref{prop:Holder-cont-deterministic-step-3} hold true $(\mu_n)$ a.s. for all $\widetilde{\gamma}<\tfrac{\overline{\alpha}}{2}-\tfrac12$
\end{proof}

\begin{lemma} \label{lem:technical-estimate}
The estimate  \eqref{eq:to_bound_the_dep_degree} holds for large enough $L$.
\end{lemma}
\begin{proof}
Letting
\[
\kappa_\ell= \frac{M^{n(1-2\overline{\alpha}+2\varepsilon)/2}}{2(\ell-n+2)^2},
\]
and applying \eqref{eq:Hoefdinng_kappa_l} for each $\ell$, a calculation analogous to \eqref{eq:corollary_of_Hoeffding} yields
\[
\PP\left(\widetilde{Y}^V_{n+1}-\widetilde{Y}^V_n>\sqrt{2}\beta_{n+1}^2 M^{-n}+\sqrt{M^{n(2\overline{\alpha}-1+2\varepsilon)}\widetilde{Y}^V_n}\right)\le C\exp(-c M^{n\varepsilon})\,.
\]
Note that because we are back in the two-dimensional situation, there is a dependency graph of bounded degree, and the martingale condition $\EE(X_Q|A_n)=0$ fails at a single cube.

Let $\mathcal{R}_n$ consist of the lines in $\mathcal{R}$ corresponding to the parameter values $u''=k M^{-2n}$,  $k\in\N$, $k\le 3 M^{2n}$. Then $\#\mathcal{R}_n\le 9 M^{2n}$ and thus
\begin{align*}
&\PP\left(\widetilde{Y}^V_{n+1}-\widetilde{Y}^V_n>\sqrt{2}\beta_{n+1}^2 M^{-n}+\sqrt{M^{n(2\overline{\alpha}-1+2\varepsilon)}\widetilde{Y}^V_n}\text{ for some }V\in\mathcal{R}_{n+1}\right)\\
&\le  C M^{2n}\exp(-c M^{n\varepsilon})\le C\exp(-c M^{\varepsilon n/2})\,.
\end{align*}
Given $V\in\mathcal{R}$ corresponding to a parameter $u''$, we can pick another line of the same type $V_0\in\mathcal{R}_{n+1}$ corresponding to $u''_0$ with $|u''-u''_0|<M^{-2(n+1)}$. The Lipschitz bound \eqref{eq:aprioriHolder} holds also (with a different constant) for the lines $x+2y=u''$, $2x+y=u''$, so we can estimate
\[
|\widetilde{Y}^V_{n+1}-\widetilde{Y}^V_n| \le | \widetilde{Y}^{V_0}_{n+1}-\widetilde{Y}^{V_0}_n | + C\beta_{n+1}^2M^{n+1}M^{-2(n+1)}.
\]
We can deduce that
\begin{equation}
\PP\left(\widetilde{Y}^V_{n+1}-\widetilde{Y}^V_n>Z_V\text{ for some }V\in\mathcal{R}\right)\le C\exp(-c M^{n\varepsilon/2})\,\label{eq:joku},
\end{equation}
where
\[
Z_V = \sqrt{2}\beta_{n+1}^2 M^{-n}+C\beta_{n+1}^2M^{n+1}M^{-2(n+1)}+\sqrt{M^{n(2\overline{\alpha}-1+2\varepsilon)}\overline{Y}^V_n}\,.
\]
If $\widetilde{X}_n\le L M^{n(2\overline{\alpha}-1+2\varepsilon)}$ (and $L$ is large enough), then
\begin{align*}
Z_V &\le C M^{n(2\overline{\alpha}-1+\varepsilon)}+\sqrt{L}M^{n(2\overline{\alpha}-1+2\varepsilon)}\\
&\le L\left( M^{(n+1)(2\overline{\alpha}-1+2\varepsilon)}-M^{n(2\overline{\alpha}-1+2\varepsilon)}\right)\,,
\end{align*}
using that $\beta_{n+1}^2\le C M^{n(2\overline{\alpha}+\varepsilon)}$. Combining this with \eqref{eq:joku} yields \eqref{eq:to_bound_the_dep_degree}.
\end{proof}

\begin{lemma} \label{lem:dependency-degree-bound}
Conditioned on $A_n$, there is a dependency graph for $(X_Q:Q\in\Q_{u,n})$ with degree at most $C\widetilde{X}_n\beta_{n}^{-2}M^{n}$, where $C$ is independent of $n, A_n$ and $u$.

Moreover, $\#\Q'_{u,n}\le C\widetilde{X}_n\beta_{n}^{-2}M^{n}$.
\end{lemma}
\begin{proof}
We define a graph $\mathcal{G}_{u,n}$ with vertex set $\Q_{u,n}$ as follows: let $Q_i=I_i\times J_i\times K_i$, $i=1,2$, with $I_i,J_i,K_i\in \Q_n$. Then there is an edge between $Q_1$ and $Q_2$ if and only if $\{ I_1,J_1,K_1\}\cap \{ I_2,J_2,K_2\}\neq\varnothing$. It is immediate from \eqref{SI} that this is indeed a dependency graph, so our task is to bound its degree.

Let $Q=I_1\times I_2\times I_3\in\Q_{u,n}$, with $I_i\in\Q_n$. If there is an edge from $Q$ to $Q'$ in $\mathcal{G}_{u,n}$, there exist $1\le i,j\le 3$ such that the $i$-th coordinate projection of $Q'$ is $I_j$. Hence, it is enough to show the following: if $Q^1,\ldots, Q^K\in\mathcal{Q}_{u,n}$ have a joint coordinate projection, then $K\le C \beta_n^{-2}M^{n}\widetilde{X}_n$. Without loss of generality, we may assume that each $Q_i$ is of the form $\widetilde{Q}^i\times I$, where $I=[a,a+M^{-n})\in\mathcal{Q}_n$ is fixed and $\widetilde{Q}^i\in\mathcal{Q}^2_n$ depends on $i$.

\begin{figure}

   \centering

\resizebox{0.7\textwidth}{!}{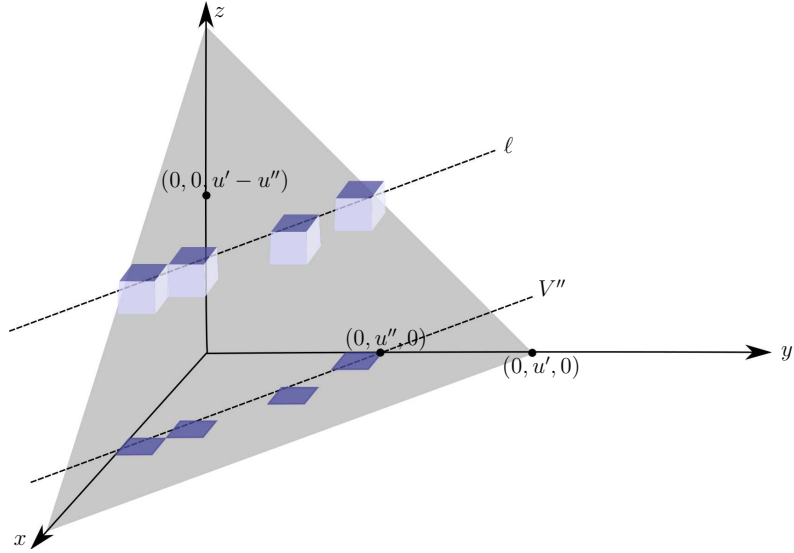}

\caption{Illustration for the application of the coarea formula. The picture shows the plane $\Pi^{-1}(u')$, the line $\ell=\{z=u'-u'', x+y=u''\}$ and its projection $V''$ onto the $(x,y)$-plane, along with the cubes $Q^1,\ldots,Q^K$ and their projections.}

\label{fig:coarea}

\end{figure}

Applying the linear coarea formula (see \cite[Lemma 1 in Section 3.4]{EvansGariepy92} and Figure \ref{fig:coarea} for an illustration), and recalling the definition of $\widetilde{Y}^{V}_n$, we estimate
\begin{align*}
K M^{-3n}&=\leb^3\left(\bigcup_{i=1}^K Q^i\right)=C\int\limits_{u'=u-3M^{-n}}^{u+3M^{-n}}\mathcal{H}^2\left(\bigcup_{i=1}^K Q^i\cap\{x+y+z=u'\}\right)\,du'\\
&=C\int\limits_{u-3M^{-n}}^{u+3M^{-n}}\int\limits_{u'-a-M^{-n}}^{u'-a}\mathcal{H}^1\left(\bigcup_{i=1}^K Q^i\cap\{z+u''=u'\,,\,x+y=u''\}\right)\,du''\,du'\\
&\le C\beta_n^{-2}\int\limits_{u'=u-3M^{-n}}^{u+3M^{-n}}\int\limits_{u''=u'-a-M^{-n}}^{u'-a}\widetilde{Y}^{V''}_n\,du''\,du'\le C\beta_n^{-2} M^{-2n}\widetilde{X}_n\,,
\end{align*}
where $V''\subset\R^2$ is the line $x+y=u''$. Hence, \eqref{eq:dep_degree} follows.

Note that here we have used the bound $\widetilde{Y}^V_n\le \widetilde{X}_n$ only for lines $V\subset\R^2$ of the form $x+y=u''$. Since the same estimate holds also for the lines $x+2y=u''$, $2x+y=u''$, a similar application of the coarea formula implies that each $\Pi^{-1}(u)$ intersects at most $C\beta_n^{-2} M^{n}\widetilde{X}_n$ cubes in $\mathcal{Q}_{u,n}$ intersecting one of the semidiagonals $\{x=y\}$, $\{x=z\}$ or $\{y=z\}$, giving the second claim.
\end{proof}

\section{Fourier decay, and restriction}
\label{sec:fourierdecay}

\subsection{Fourier decay and Salem sets and measures}

Let $\mu$ be a compactly supported probability measure on $\R^d$. The \emph{Fourier transform} of $\mu$ is defined as
\[
\widehat{\mu}(\xi) = \int e^{-2\pi i \,x\cdot \xi} \,d\mu(x)\,,\quad\xi\in\R^d\,.
\]
The speed of decay of $\widehat{\mu}(\xi)$ as $|\xi|\to\infty$ (if any) gives important information about $\mu$. Very roughly speaking, slow or no decay indicates that $\mu$ has ``arithmetic structure'' while fast decay indicates ``pseudo-randomness''. In more quantitative terms, a first question is whether there is any power decay
\begin{equation} \label{eq:power-fourier-decay}
|\widehat{\mu}(\xi)| \le C_\sigma (1+|\xi|)^{-\sigma/2}
\end{equation}
for some $\sigma>0$, and if so, what is the largest such $\sigma$. The reason for looking at $\sigma/2$ (rather than $\sigma$) comes from the following universal upper bound: if the $t$-energy
\[
I_t(\mu) = \iint \frac{d\mu(x)d\mu(y)}{|x-y|^t}=\infty,
\]
then \eqref{eq:power-fourier-decay} cannot hold for any $t<\sigma<d$. This is due to the expression of the energy in terms of the Fourier transform:
\begin{equation} \label{eq:energy-fourier}
I_t(\mu) = C_{t,d} \int |\xi|^{t-d} |\widehat{\mu}(\xi)|^2\,dx.
\end{equation}
See \cite[Lemma 12.12]{Mattila95} for the proof of this identity. In particular, if the topological support of $\mu$ has dimension $t$, then \eqref{eq:power-fourier-decay} cannot hold for any $\sigma>t$. These observations lead to the following definition: the \emph{Fourier dimension} of a measure $\mu$ on $\R^d$ is
\[
\dim_F(\mu) = \sup\left\{\sigma: |\widehat{\mu}(\xi)| \le C_\sigma (1+|\xi|)^{-\sigma/2} \right\},
\]
and the Fourier dimension of a set $A\subset\R^d$ is
\[
\dim_F(A) = \sup\{ \dim_F(\mu): \mu(A)=1\},
\]
where the supremum is over Borel probability measures. See \cite{EPS15} for a discussion of the properties of Fourier dimension and some variants. Our previous discussion shows that one always has an inequality $\dim_F(A)\le \dim_H(A)$. Sets for which $\dim_F(A)=\dim_H(A)$ are called \emph{Salem sets}, as it was Salem \cite{Salem51} who first constructed examples of such sets. Salem sets abound as random sets (see e.g. \cite{LabaPramanik09} and references there), but few deterministic fractal Salem sets are known (curved manifolds such as the sphere are Salem sets - this is proved with standard stationary phase methods).

Salem sets, therefore, should be thought of as pseudo-random in some sense. The next well-known lemma gives a concrete manifestation of this. Recall that the arithmetic sum of two sets $A,B\subset\R^d$ is $A+B=\{x+y:x\in A,y\in B\}$.
\begin{lemma} \label{lem:Salem-sumset}
Let $A\subset\R^d$ be a Salem set. Then for any Borel set $B\subset\R^d$,
\[
\dim_H(A+B) = \min(\dim_H(A)+\dim_H(B),d).
\]
Moreover, if $\dim_H(A)+\dim_H(B)>d$, then $A+B$ has positive Lebesgue measure.
\end{lemma}
\begin{proof}
Pick $t<\dim_H(A), t'<\dim_H(B)$. By the definition of Salem set, there is a Borel probability measure $\mu$ such that $\mu(A)=1$ and $\dim_F(\mu)>t$. By Frostman's Lemma (see e.g. \cite[Theorem 8.8]{Mattila95}), we can also find a Borel probability measure $\nu$ with $\nu(B)=1$, such that $I_{t'}(\mu)<+\infty$. Note that $\mu*\nu(\R^d\setminus(A+B))=0$. Using the expression for the energy in terms of the Fourier transform, Equation \eqref{eq:energy-fourier}, we obtain
\begin{align*}
I_{t+t'}(\mu*\nu)  &= C_{t,t',d} \int |\xi|^{t+t'-d} |\widehat{\mu*\nu}(\xi)|^2 \,d\xi\\
&=  C_{t,t',d} \int |\xi|^{t+t'-d} |\widehat{\mu}(\xi)|^2 |\widehat{\nu}(\xi)|^2 \,d\xi\\
&\le C_{t,t',d} C'_{t} \int |\xi|^{t+t'-d}  |\xi|^{-t} |\widehat{\nu}(\xi)|^2 \,d\xi\\
&\le C_{t,t',d} C'_{t} I_{t'}(\nu) <+\infty.
\end{align*}
If $t+t'<d$, the finiteness of the energy implies that $\dim_H(A+B)\ge t+t'$ (see e.g. \cite[Theorem 8.7]{Mattila95}), while if $t+t'\ge d$, then the above calculation together with Parseval's Theorem show that $\mu*\nu$ has a density in $L^2$, so $A+B$ has positive Lebesgue measure. Letting $t\uparrow \dim_H(A), t'\uparrow\dim_H(B)$ finishes the proof.
\end{proof}

\subsection{Fourier decay of the random measures}

There is no universally agreed definition of \emph{Salem measures} (partly because there are many notions of dimension one could use). However, it is clear from the previous discussion that if a measure satisfies $|\widehat{\mu}(\xi)|\le C_\sigma |\xi|^{\sigma/2}$, then $\sigma\le\dim_H(\supp(\mu))$. We will see that for the class of random measures studied in this article, this holds for any $\sigma<d-\overline{\alpha}$ when $d=1$ or $2$. In particular, this implies that the random sets $A=\supp(\mu)$ are Salem sets provided that $\underline{\alpha}=\overline{\alpha}=\alpha$ (so that $\dim_H A=d-\alpha$ a.s., recall Lemma \ref{lem:asdim}).

\begin{thm} \label{thm:Salem}
Suppose $(\mu_n)$ satisfies \eqref{mu0}--\eqref{SI} and $\overline{\alpha}\in[d-2,d]$, where $\overline{\alpha}$ is as in \eqref{eq:def-alpha}. Then almost surely the following holds for the limit measure $\mu$: for any $\sigma<d-\overline{\alpha}$, there is a constant $C=C_\sigma>0$ such that
\[
|\widehat{\mu}(\xi)| \le C\, |\xi|^{-\sigma/2}\quad\text{for all }\xi\neq 0.
\]
\end{thm}
Theorem \ref{thm:Salem} was first proved in \cite[Theorem 14.1]{ShmerkinSuomala15} (the class of measures there obeys slightly different assumptions, but the changes required to obtain Theorem \ref{thm:Salem} are very minor). The restriction $\overline{\alpha}\in [d-2,d]$ might appear mysterious, but as observed in \cite[Remarks 14.2]{ShmerkinSuomala15}, it is sharp: the Fourier dimension of $\mu$ can never exceed $2$, due to the $M$-adic structure of the construction, which forces the principal projections to be discontinuous. Of course, in dimensions $d=1,2$ the assumption $\overline{\alpha}\in [0,2]$ is vacuous. We also note that, as a special case, the theorem shows that fractal percolation limit sets are Salem sets, so long as they have dimension at most $2$ (which again is a necessary condition). It also gives many examples of Salem sets which are also Ahlfors-regular.

Once again, we will only present the proof of the theorem in the case of dimension $d=1$. The proof in the general case is very similar, but slightly more technical. The ideas of the proof are inspired by a related construction of {\L}aba and Pramanik \cite{LabaPramanik09}, and there are strong parallels with the proof of Theorem \ref{thm:self-convolutions} in the simpler double-convolution case.

\begin{proof}[Proof of Theorem \ref{thm:Salem} in the case $d=1$]
Fix $\sigma<1-\overline{\alpha}$. It is enough to show that $\widehat{\mu}(k) \le C_\sigma |k|^{-\sigma/2}$ for $k\in\Z$, as this implies decay over real frequencies, see \cite[Lemma 9A4]{Wolff03}.

For fixed $k$, we have
\[
\widehat{\mu}_{n+1}(k)-\widehat{\mu}_n(k)=\sum_{Q\in\Q_n}X_Q\,,
\]
where
\[
X_Q =\int_Q (\mu_{n+1}(x)-\mu_n(x)) \exp(-2\pi i k x) \,dx\,.
\]
Then $|X_Q|\le C \beta_{n+1} M^{-n}$. Let $N_n$ be the number of cubes making up $A_n$. Using Lemma \ref{lem:HoeffdingJanson} (with $\Delta=0$) for the real and imaginary parts of $X_Q$ and property \eqref{Madic}, we obtain
\begin{align*}
\PP\left(|\widehat{\mu}_{n+1}(k)-\widehat{\mu}_n(k)| > M^{-\sigma n/2}\|\mu_n\|^{1/2}\,|\,A_n\right) &\le
C\exp\left(-c M^{(2-\sigma)n}\|\mu_n\|\beta_{n+1}^{-2}N_n^{-1}\right)\\
&\le C \exp(-c M^{(1-\sigma)n}\beta_n\beta_{n+1}^{-2})\,.
\end{align*}
Since $\liminf_n\log_M\beta_{n}/\log_M\beta_{n+1}^{-2}\ge-\overline{\alpha}$ and $\sigma<1-\overline{\alpha}$, we get
\begin{equation} \label{eq:Salem-large-dev}
\PP\left(|\widehat{\mu}_{n+1}(k)-\widehat{\mu}_n(k)| > M^{-\sigma n/2}\|\mu_n\|^{1/2} \text{ for some } |k|< M^{n+1}\right) \le \e_n\,,
\end{equation}
where $\e_n\le C\, M^n \exp(-c M^{(1-\sigma)n}\beta_{n}\beta_{n+1}^{-2})$ is summable.

Let $Q\in\mathcal{Q}_{n+1}$. Since the endpoints of the interval $Q$ are of the form $j M^{-n-1}$, it follows that for $|k|<M^{n+1},0\neq\ell\in\mathbb{Z}$,
\[
\widehat{\mathbf{1}}_Q(k+M^{n+1}\ell) = \frac{k}{k+ M^{n+1}\ell}\, \widehat{\mathbf{1}}_Q(k)\,.
\]
Since $\widehat{\mu}_{n+1}$ is a linear combination of the functions $\widehat{\mathbf{1}}_Q$, the same relation holds between $\widehat{\mu}_{n+1}(k+M^{n+1}\ell)$ and $\widehat{\mu}_{n+1}(k)$ (and this holds also for $\widehat{\mu}_{n}$ in place of $\widehat{\mu}_{n+1}$). Fix $k'$ with $|k'|\ge M^{n+1}$, and write $k'=k+M^{n+1}\ell$, where $|k|<M^{n+1}$, $k\ell\ge 0$, and $\ell\neq 0$. For such $k'$, we have
\begin{align*}
|\widehat{\mu}_{n+1}(k')-\widehat{\mu}_n(k')| &\le \frac{|k|}{|k+ M^{n+1}\ell|}|\widehat{\mu}_{n+1}(k)-\widehat{\mu}_n(k)|\\
&< \frac{M^{n+1}}{|k'|}|\widehat{\mu}_{n+1}(k)-\widehat{\mu}_n(k)|\,.
\end{align*}
Combining this with \eqref{eq:Salem-large-dev}, we arrive at the following key fact: $\PP(E_n)<\e_n$, where $E_n$ is the event
\[
\left|\widehat{\mu}_{n+1}(k)-\widehat{\mu}_n(k)\right|>
\|\mu_n\|^{1/2}\min\left(1,\frac{M^{n+1}}{|k|} \right) M^{-\sigma n/2}\text{ for some }k\in\Z\,.
\]
Since $\e_n$ was summable, and $\|\mu_n\|$ is a.s. bounded, it follows from the Borel-Cantelli Lemma that a.s. there are $C$ and $n_0$ such that
\[
\left|\widehat{\mu}_{n+1}(k)-\widehat{\mu}_n(k)\right|  \le C \min\left(1,\frac{M^{n+1}}{|k|} \right) M^{-\sigma n/2}\quad\text{for all }k\in\Z,n\ge n_0\,.
\]
Thus, choosing $n_1\in\N$ such that $M^{n_1}\le|k|<M^{n_1+1}$, and telescoping, we have
\begin{equation}\label{eq:telescope}
\begin{split}
\left|\widehat{\mu}_m(k)-\widehat{\mu}_{n_0}(k)\right|&\le\sum_{n_0\le n\le n_1}2CM^{n(1-\sigma/2)}|k|^{-1} +\sum_{\max(n_1,n_0)<n\le m}C M^{-n\sigma/2}\\
&\le C\,|k|^{-\sigma/2}
\end{split}
\end{equation}
for all $k\in\mathbb{Z},m\ge n_0$. Noting that $|\widehat{\mu}_{n_0}(k)|\le C_{n_0}|k|^{-1}$ and letting $m\to\infty$ finishes the proof.
\end{proof}

\subsection{Restriction for fractal measures}

Given a measure $\mu$ on $\R^d$ the \emph{restriction problem} for $\mu$ consists in determining for which values of $p,q$ there is an estimate
\begin{equation} \label{eq:restriction}
\|\widehat{fd\mu}\|_{L^p(\mathcal{L}^d)} \le C_{p,q,\mu} \|f\|_{L^q(\mu)}.
\end{equation}
The name comes from the dual formulation of \eqref{eq:restriction}, namely
\[
\|\widehat{f}\|_{L^{q'}(\mu)} \le C'_{p,q,\mu} \|f\|_{L^{p'}(\mathcal{L}^d)}.
\]
In other words, the goal is to understand for which functions $f$ it is meaningful to restrict $\widehat{f}$ to the support of $\mu$. Classically, the restriction problem has been studied for surfaces such as the sphere or the paraboloid. In the case of the sphere and $q=2$, a famous theorem of Stein and Tomas gives the sharp range of $p$ for which \eqref{eq:restriction} holds: $p\ge (2d+2)/(d-1)$. The problem of finding all the pairs $(p,q)$ is still open and is closely connected to other well-known problems such as the Kakeya problem. The study of the restriction problem for fractal measures was initiated by Mockenhaupt \cite{Mockenhaupt00} (see also \cite{Mitsis02}), who extended the Stein-Tomas argument to very general measures satisfying suitable mass and Fourier decay: if a measure $\mu$ on $\R^d$ satisfies
 \begin{align}
\mu(B(x,r))&\le C_1\,r^s\,, \label{eq:decay-mass}\\
|\widehat{\mu}(\xi)| &\le  C_2 \,(1+|\xi|)^{-\sigma/2} \,,\label{eq:decay-FT}
\end{align}
then \eqref{eq:restriction} with $q=2$ holds whenever
\[
p >  p_{s,\sigma,d} = \frac{2(2d-2s+\sigma)}{\sigma}\,.
\]
Bak and Seeger \cite{BakSeeger11} proved that \eqref{eq:restriction} also holds at the endpoint $p=p_{s,\sigma,d}$. (Note that taking $s=\sigma=d-1$, this recovers the Stein-Tomas estimate in the case of the sphere). Hambrook and {\L}aba \cite{HambrookLaba13} (see also \cite{Chen16} for a generalization) constructed, for a dense set of $t\in [0,1]$, measures $\mu$ on the real line satisfying \eqref{eq:decay-mass} and \eqref{eq:decay-FT} for $s,\sigma$ arbitrarily close $t$, and supported on sets of Hausdorff dimension $t$, for which the restriction estimate \eqref{eq:restriction}  does not hold for any $p<p_{t,t,1}$. This shows that in general the result of Mockenhaupt, Bak and Seeger is sharp also for fractal measures. However, they left open the problem of whether one can go beyond this range for \emph{some} fractal measures of this kind. This question was explicitly asked in \cite{Laba14}.

A different general restriction theorem based on convolution powers was proved by X. Chen \cite{Chen14}. As a special case of his main result, he showed that if the $n$-th convolution power $\mu^{*n}$ of a measure $\mu$ on $\R^d$ has a bounded density, then \eqref{eq:restriction} holds whenever $p\ge 2n$ and $q\ge p/(p-n)$ (in particular, for $q=2$). Theorem \ref{thm:self-convolutions} provides a rich class of random measures to which this result applies. Moreover, Theorem \ref{thm:Salem} shows that many of these measures can also be chosen to have essentially optimal Fourier decay, and using our class of examples (ii) it is also possible to get Ahlfors-regular examples. When the dimension of the (support of the) measure lies in $(1/2,2/3)$, the range in the restriction theorem of Chen is $p\ge 4$ when $q= 2$, and this is a larger range than that coming from Mockenhaupt's Theorem, since $p_{t,t,1}=4/t-2>4$ when $t<2/3$. Hence, there is a large class of random measures $\mu$ on the real line supported on sets of any dimension $t\in(1/2,2/3)$ such that:
\begin{itemize}
\item $\mu$ satisfies \eqref{eq:decay-mass} and \eqref{eq:decay-FT} for $s,\sigma$ arbitrarily close to $t$,
\item The restriction estimate \eqref{eq:restriction} holds for $q=2$ and all $p\ge 4>p_{t,t,1}$.
\end{itemize}
This class includes fractal percolation: it can be shown that the fractal percolation measure satisfies \eqref{eq:decay-mass} for all $s$ smaller than the a.s. dimension. This partially answers the question of I. {\L}aba. In fact, by adapting the proof of Theorem \ref{thm:self-convolutions}, it is easy to show the existence of measures supported on sets of dimension exactly $1/2$ for which the above is true. The significance of $1/2$ is that the range $p\ge 4$ (with $q=2$) is sharp in this case, as can be seen from dimensional considerations (see \cite{Chen14}). In general, this method gives measures supported on sets of dimension $1/n$, $n\in\N$, for which the range of exponents in \eqref{eq:restriction} with $q=2$ is sharp. The precise connection between restriction and Hausdorff dimension is still not fully understood.

We remark that Chen and Seeger \cite{ChenSeeger15} constructed random measures with similar properties, but using entirely different methods. Their result works for any ambient space and dimension of the support of the measure (while ours only applies to measures of dimension up to $2$, due to Theorem \ref{thm:Salem}). On the other hand, our construction includes Ahlfors-regular examples, and well-known models such as fractal percolation.

\section{Arithmetic structure}
\label{sec:arithmeticstructure}

A basic problem in additive combinatorics is to understand what properties of a set $A\subset \{1,\ldots,N\}$ imply that $A$ has an arithmetic progression of a given length $k$. If $|A|\ge \delta N$, and $N$ is sufficiently large in terms of $k$, then Szemer\'{e}di's famous theorem asserts that $A$ does contain a progression of length $k$. On the other hand, as we will see below, for every $\e>0$ and large enough $N$, there are sets of size $N^{1-\e}$ that do not even contain progressions of length $3$.

It is natural to ask similar questions for subsets of $[0,1]$. A simple application of the Lebesgue density theorem shows that if $A$ has positive Lebesgue measure, then $A$ contains arithmetic progressions of any length and, more generally, contains an homothetic copy of any finite subset of $\R$. On the other hand, Keleti \cite{Keleti98} constructed a compact set of Hausdorff dimension $1$ that does not contain any rectangles $x,x+r,y,y+r$, and in particular contains no progressions of length $3$. In a different direction,  Davis, Marstrand and Taylor \cite{DMT59} constructed a compact set of zero Hausdorff dimension which contains a similar copy of all finite sets. It then appears that Hausdorff dimension by itself is insufficient to detect the presence, or lack thereof, of finite patterns.

In the discrete setting, it is well-known that Fourier uniformity is enough to guarantee the existence of three-term arithmetic progressions even for rather sparse (although not extremely sparse) sets. It is natural to ask whether a similar result holds for subsets of $[0,1]$. I. {\L}aba and M. Pramanik \cite{LabaPramanik09} proved a result in this direction:
\begin{thm}[{\L}aba and Pramanik] \label{thm:LabaPramanik}
Given $C_1, C_2>0, \sigma>0$, there exists $\e_0=\e_0(C_1,C_2,\sigma)>0$ such that the following holds: if  $s>1-\e_0$, and $\mu$ is a measure on $[0,1]$ such that
\begin{enumerate}
\item[(i)] $\mu(x,x+r) \le C_1 r^s$ for all $x\in \R/\Z$ and all $r\in (0,1)$,
\item[(ii)] $|\wh{\mu}(\xi)| \le C_2 |\xi|^{-\sigma/2}$ for all $\xi\neq 0$,
\end{enumerate}
then $\supp(\mu)$ contains a $3$-term arithmetic progression.
\end{thm}
In fact, the original result from \cite{LabaPramanik09} requires $\sigma>2/3$ in the Fourier decay assumption, while the relaxation to any $\sigma>0$ is due to \cite[Theorem 10.1]{HLP15}, where generalizations to certain polynomial patterns and higher dimensions are also obtained. We emphasize that the mass decay must be fast enough not just in terms of the Fourier decay, but also in terms of the constants $C_1, C_2$. This makes the assumptions hard to verify in practice. This then leads to the question of whether the mass and Fourier decay conditions (i), (ii) with sufficiently large $s,\sigma<1$ (independent of the constants) are enough to guarantee the presence of arithmetic progressions in the support of $\mu$. Results such as Lemma \ref{lem:Salem-sumset} show that, for certain problems, it is only the polynomial exponent of decay that matters, and not the constant. This suggests the question of whether the dependence of $\e_0$ on the constants (and particularly on $C_1$) is really needed in Theorem \ref{thm:LabaPramanik}. Recently, the first author \cite{Shmerkin16} used a random construction closely related to the class studied in this note to show that the answer is yes, in a strong sense:
\begin{thm} \label{thm:Salem-no-progressions}
For every $s\in (0,1)$ there exists a Borel probability measure $\mu$ on $[0,1]$ such that:
\item[(i)] $\mu(x,x+r) \le C r^s$ for all $x\in [0,1]$ and all $r\in (0,1)$, and some $C>0$ that depends on $s$.
\item[(ii)] For all $\sigma\in(0,s)$, there exists a constant $C_\sigma>0$ such that $|\wh{\mu}(\xi)| \le C_\sigma |\xi|^{-\sigma/2}$ for all $\xi\neq 0$,
\item[(iii)] The topological support of $\mu$ does not contain any arithmetic progressions of length $3$.
\end{thm}
In fact, an even stronger statement is obtained: one can additionally choose $\mu$ to either be Ahlfors-regular, or to satisfy the first condition for \emph{all} $s<1$ (with the constant $C$ depending on $s$).

\begin{proof}[Proof of Theorem \ref{thm:Salem-no-progressions}]
We will construct measures supported on sets of dimension arbitrarily close to $1$, and leave to the reader the small modifications needed to construct measures of arbitrary dimension, or see \cite{Shmerkin16}.

Let $\Z_M$ be the classes of residues modulo $M$. A classical example due to Behrend \cite{Behrend46} shows the existence of a set $E\subset \{0,1,\ldots,M-1\}$ with no $3$-term arithmetic progressions and size at least $\exp(-c\sqrt{\log M})M$. An easy argument (see \cite{Shmerkin16} for details) shows that one can do the same with $E\subset\Z_M$, where moreover $M$ is even and all the elements of $E$ are even as well. In particular, given $\e>0$, we can fix a large enough even number $M$ and a set $E\subset\Z_M$ with no progressions, all elements even, and $\#E\ge M^{1-\e}$.

Now given $Q\in\mathcal{Q}_n$ and the base $M$, we can label the set of offspring intervals $\mathcal{S}(Q)$ by $\{0,1,\ldots,M-1\}$, and in turn identify this with $\Z_M$. Let $\{ a_Q:Q\in\mathcal{Q}_n,n\in\N\}$ be i.i.d. random variables chosen uniformly in $\{0,1,\ldots,M-1\}$, and set $S_Q=E+a_Q \bmod M$. We can then carry out the construction described in Section \ref{sec:themodel}, to obtain a sequence $(\mu_n)$ satisfying \eqref{mu0}--\eqref{SI} with $\beta_n = \#E^n$. Note that $E+a_Q$ does not contain progressions as a subset of $\Z_M$, and hence also a subset of $\{0,1,\ldots,M-1\}$.

Let $\mu$ be the limit measure. Since $\#S_Q=\#E$ is deterministic and constant, it is easy to see that (i) holds with $s=\log_M\#E\in (1-\e,1)$. The Fourier decay (ii) is direct from Theorem \ref{thm:Salem}.

Let us see, then, that $A=\supp(\mu)$ does not contain any $3$-term progressions. Suppose to the contrary that $\{ y_1<y_2<y_3\}\subset A$ is such a progression. Let $Q$ be a minimal $M$-adic interval with $\{ y_1,y_2,y_3\}\subset Q$. By the self-similarity of the construction, we may assume that $Q=[0,1)$. Write $y_i=x_i+\delta_i$, where $x_i$ is the left-endpoint of the $M$-adic interval containing $y_i$, and $\delta_i\in [0,1/M)$. From $y_2=\tfrac12(y_1+y_3)$, we get
\[
\frac12{(\delta_1+\delta_3)} -\delta_2 = x_2 - \frac12(x_1+x_3)
\]
The left hand-side is at most $1/M$. On the other hand, by construction the right-hand size is not zero (since $x_1,x_2,x_3$ are not all equal and do not form an arithmetic progression). Also, the numbers $M x_i$ are either all even or all odd, so the right-hand side is at least $1/M$. Hence the only option is that both the left and right-hand sides are equal to $1/M$, but this forces all of the $y_i$ to be endpoints of $M$-adic intervals, which almost surely does not happen (since a fixed point has probability zero of belonging to $A$, and there are countably many such endpoints). This contradiction finishes the proof.
\end{proof}

We finish the article by commenting on the opposite problem of finding sets with many patterns. As mentioned before, it follows from \cite{DMT59} that there are compact subsets of the real line of zero Hausdorff dimension which contain an homothetic copy of all finite sets. However, such a set necessarily has packing and box-counting dimension $1$:
\begin{lemma}  \label{lem:packingdim-patterns}
If $A\subset\R$ contains a homothetic image of all $m$-element sets, then $\dim_H(A^m)\ge m-2$, and $\dim_P(A)\ge 1-2/m$, where $\dim_P$ denotes packing dimension.
\end{lemma}
\begin{proof}
Let $X_m=\{ (x_1,\ldots,x_m)\in\R^{m}: x_1<\ldots<x_m\}$. Consider the map $\Pi:X_m\to \R^{m-2}$,
\[
(x_1,\ldots,x_m)\mapsto  \frac{1}{x_2-x_1}(x_3-x_1,\ldots,x_m-x_1).
\]
This map is locally Lipschitz so does not increase Hausdorff dimension. On the other hand, $A$ contains an homothetic image of $(0,1,t_3,\ldots,t_m)$ with $1<t_3<\ldots<t_m$ if and only if $(t_3,\ldots,t_m)\in\Pi(A^m\cap X_m)$. Hence, if $A$ contains an homothetic copy of all $m$-element sets, we must have
\[
\left(\dim_P(A)\right)^m \ge \dim_P(A^m) \ge \dim_H(A^m)\ge m-2,
\]
where the left-most inequality is a well-known property of packing dimension, see e.g. \cite[Theorem 3]{Tricot82}. The claim follows.
\end{proof}

This lemma suggests the following question: how small can the \emph{packing} dimension of a set containing an homothetic image of all $m$-element sets be? Using methods similar to those used to prove Theorem \ref{thm:self-convolutions}, but with additional technical difficulties, we are able to prove the following: let  $A\subset [0,1]$ be the fractal percolation set constructed with parameters $p$ and $M$, and write $s=1+\log_Mp$ for the almost sure dimension. If $s>1-2/m$, with $m\ge 3$, then almost surely on $A\neq \varnothing$, the set $A$ contains an homothetic copy of all $m$-element sets. This result will appear in \cite{ShmerkinSuomala16}. To understand the analogy with Theorem \ref{thm:self-convolutions} for three-fold convolutions, recall that in the proof it was key to understand the intersections of $A_n\times A_n\times A_n$ with $(\Pi')^{-1}(u)$, where $\Pi':\R^3\to\R, (x,y,z)\to x+y+z$. In this case, one needs to do a similar study for the fibers $\Pi^{-1}(t_3,\ldots,t_m)$, where $\Pi$ is as in the proof of Lemma \ref{lem:packingdim-patterns}. Additional complications are caused by the non-linearity of $\Pi$. Also note that Theorem \ref{thm:Salem-no-progressions} shows that one cannot hope to have a similar result of the same generality as Theorem \ref{thm:self-convolutions}.


\begin{thebibliography}{10}

\bibitem{BakSeeger11}
Jong-Guk Bak and Andreas Seeger.
\newblock Extensions of the {S}tein-{T}omas theorem.
\newblock {\em Math. Res. Lett.}, 18(4):767--781, 2011.

\bibitem{Behrend46}
F.~A. Behrend.
\newblock On sets of integers which contain no three terms in arithmetical
  progression.
\newblock {\em Proc. Nat. Acad. Sci. U. S. A.}, 32:331--332, 1946.

\bibitem{ChenSeeger15}
X.~Chen and A.~Seeger.
\newblock Convolution powers of {S}alem measures with applications.
\newblock Preprint, available at \texttt{http://arxiv.org/abs/1509.00460},
  2015.

\bibitem{Chen14}
Xianghong Chen.
\newblock A {F}ourier restriction theorem based on convolution powers.
\newblock {\em Proc. Amer. Math. Soc.}, 142(11):3897--3901, 2014.

\bibitem{Chen16}
Xianghong Chen.
\newblock Sets of {S}alem type and sharpness of the {$L\sp 2$}-{F}ourier
  restriction theorem.
\newblock {\em Trans. Amer. Math. Soc.}, 368(3):1959--1977, 2016.

\bibitem{DMT59}
Roy~O. Davies, J.~M. Marstrand, and S.~J. Taylor.
\newblock On the intersections of transforms of linear sets.
\newblock {\em Colloq. Math.}, 7:237--243, 1959/1960.

\bibitem{EPS15}
Fredrik Ekstr{\"o}m, Tomas Persson, and J{\"o}rg Schmeling.
\newblock On the {F}ourier dimension and a modification.
\newblock {\em J. Fractal Geom.}, 2(3):309--337, 2015.

\bibitem{EvansGariepy92}
Lawrence~C. Evans and Ronald~F. Gariepy.
\newblock {\em Measure theory and fine properties of functions}.
\newblock Studies in Advanced Mathematics. CRC Press, Boca Raton, FL, 1992.

\bibitem{HambrookLaba13}
Kyle Hambrook and Izabella {\L}aba.
\newblock On the sharpness of {M}ockenhaupt's restriction theorem.
\newblock {\em Geom. Funct. Anal.}, 23(4):1262--1277, 2013.

\bibitem{HLP15}
Kevin Henriot, Izabella {\L}aba, and Malabika Pramanik.
\newblock On polynomial configurations in fractal sets.
\newblock Preprint, available at \texttt{http://arxiv.org/abs/1511.05874},
  2015.

\bibitem{Janson04}
Svante Janson.
\newblock Large deviations for sums of partly dependent random variables.
\newblock {\em Random Structures Algorithms}, 24(3):234--248, 2004.

\bibitem{Keleti98}
Tam{\'a}s Keleti.
\newblock A 1-dimensional subset of the reals that intersects each of its
  translates in at most a single point.
\newblock {\em Real Anal. Exchange}, 24(2):843--844, 1998/99.

\bibitem{Korner08}
Thomas K{\"o}rner.
\newblock On a theorem of {S}aeki concerning convolution squares of singular
  measures.
\newblock {\em Bull. Soc. Math. France}, 136(3):439--464, 2008.

\bibitem{Laba14}
Izabella {\L}aba.
\newblock Harmonic analysis and the geometry of fractals.
\newblock {\em Proceedings of the 2014 International Congress of
  Mathematicians}, To appear, 2014.

\bibitem{LabaPramanik09}
Izabella {\L}aba and Malabika Pramanik.
\newblock Arithmetic progressions in sets of fractional dimension.
\newblock {\em Geom. Funct. Anal.}, 19(2):429--456, 2009.

\bibitem{LyonsPeresbook}
R.~Lyons and Y.~Peres.
\newblock {\em Probability on Trees and Networks}.
\newblock Cambridge University Press.
\newblock 2016. In preparation. Current version available at
  \texttt{http://mypage.iu.edu/\string~rdlyons/}.

\bibitem{Mattila95}
Pertti Mattila.
\newblock {\em Geometry of sets and measures in {E}uclidean spaces}, volume~44
  of {\em Cambridge Studies in Advanced Mathematics}.
\newblock Cambridge University Press, Cambridge, 1995.
\newblock Fractals and rectifiability.

\bibitem{Mitsis02}
Themis Mitsis.
\newblock A {S}tein-{T}omas restriction theorem for general measures.
\newblock {\em Publ. Math. Debrecen}, 60(1-2):89--99, 2002.

\bibitem{Mockenhaupt00}
G.~Mockenhaupt.
\newblock Salem sets and restriction properties of {F}ourier transforms.
\newblock {\em Geom. Funct. Anal.}, 10(6):1579--1587, 2000.

\bibitem{PeresRams14}
Yuval Peres and Micha{\l} Rams.
\newblock Projections of the natural measure for percolation fractals.
\newblock {\em Israel J. of Math.}, 2014.
\newblock To appear, available at \texttt{http://arxiv.org/abs/1406.3736}.

\bibitem{RamsSimon14}
Micha{\l} Rams and K{\'a}roly Simon.
\newblock The dimension of projections of fractal percolations.
\newblock {\em J. Stat. Phys.}, 154(3):633--655, 2014.

\bibitem{RamsSimon15}
Micha{\l} Rams and K{\'a}roly Simon.
\newblock Projections of fractal percolations.
\newblock {\em Ergodic Theory Dynam. Systems}, 35(2):530--545, 2015.

\bibitem{Salem51}
R.~Salem.
\newblock On singular monotonic functions whose spectrum has a given
  {H}ausdorff dimension.
\newblock {\em Ark. Mat.}, 1:353--365, 1951.

\bibitem{Shmerkin16}
Pablo Shmerkin.
\newblock Salem sets with no arithmetic progressions.
\newblock {\em Int. Math. Res. Not. IMRN}, to appear, 2016.
\newblock available at \texttt{http://arxiv.org/abs/1510.07596}.

\bibitem{ShmerkinSuomala15b}
Pablo Shmerkin and Ville Suomala.
\newblock Sets which are not tube null and intersection properties of random
  measures.
\newblock {\em J. Lond. Math. Soc. (2)}, 91(2):405--422, 2015.

\bibitem{ShmerkinSuomala15}
Pablo Shmerkin and Ville Suomala.
\newblock Spatially independent martingales, intersections, and applications.
\newblock {\em Mem. Amer. Math. Soc.}, to appear, 2015.
\newblock Available at \texttt{http://arxiv.org/abs/1409.6707}.

\bibitem{ShmerkinSuomala16}
Pablo Shmerkin and Ville Suomala.
\newblock Patterns in random fractals.
\newblock In preparation, 2016.

\bibitem{SimonVago14}
K{\'a}roly Simon and Lajos V{\'a}g{\'o}.
\newblock Projections of {M}andelbrot percolation in higher dimensions.
\newblock In {\em Fractals, wavelets, and their applications}, volume~92 of
  {\em Springer Proc. Math. Stat.}, pages 175--190. Springer, Cham, 2014.

\bibitem{Tricot82}
Claude Tricot, Jr.
\newblock Two definitions of fractional dimension.
\newblock {\em Math. Proc. Cambridge Philos. Soc.}, 91(1):57--74, 1982.

\bibitem{Wolff03}
Thomas~H. Wolff.
\newblock {\em Lectures on harmonic analysis}, volume~29 of {\em University
  Lecture Series}.
\newblock American Mathematical Society, Providence, RI, 2003.
\newblock With a foreword by Charles Fefferman and preface by Izabella {\L}aba,
  Edited by {\L}aba and Carol Shubin.


\end{thebibliography}
\end{document}